\pdfoutput=1 

\documentclass{article}%
\usepackage{amsthm,amsmath,amssymb}
\usepackage{todonotes}
\usepackage{amsmath}
\usepackage{amsfonts}
\usepackage{amssymb}
\usepackage{graphicx}
\usepackage{subcaption}
\usepackage[top=4cm,bottom=4cm,left=4cm,right=4cm]{geometry}
\usepackage{commath}%
\usepackage{hyperref}
\providecommand{\U}[1]{\protect\rule{.1in}{.1in}}
\newtheorem{theorem}{Theorem}

\newtheorem{definition}[theorem]{Definition}

\newtheorem{proposition}[theorem]{Proposition}
\newtheorem{remark}[theorem]{Remark}

\newcommand{\C}{\mathbb{C}}
\newcommand{\D}{\mathbb{D}}
\renewcommand{\Re}{\operatorname{Re}}

\begin{document}

\title{Continuation of relative equilibria in the $n$--body
problem to spaces of constant curvature}
\author{A. Bengochea, C. Garc\'ia-Azpeitia, E. P\'erez-Chavela, P. Roldan}
\maketitle

\begin{abstract}
	We prove that all non-degenerate relative equilibria of the planar
	Newtonian $n$--body problem can be continued to spaces of constant
	curvature $\kappa$, positive or negative, for small enough values of this
	parameter.
	We also compute the extension of some classical relative
	equlibria to curved spaces using numerical continuation. In particular, we
	extend Lagrange's triangle configuration with different masses to both
	positive and negative curvature spaces.

Keywords: relative equilibria, spaces of constant curvature, curved $n$--body problem, Palais slice coordinates.

Subjclass[2020] {70F10, 70F07, 70F15, 34A25}

\end{abstract}

\section{Introduction}

The curved $n$--body problem is a natural extension of the planar Newtonian $%
n$--body problem to surfaces of nonzero constant curvature. The origin of this
problem is in the Kepler and the two body problem on curved spaces,
first studied independently by Lobachevsky and Bolyai (co-discoverers of the
first non-Euclidean geometries), in the 1830's. Since then, many authors
have worked on this problem, developed their original geometric ideas and
expressed them in analytical and modern language; among others, 
E. Shering, H. Liebmann, P. Serr\'e, and many names from the
Russian school (see \cite{Diacu} and \cite{Bor2} for a nice historical
review on this problem).

In \cite{DPS12}, the Kepler and the two--body
problem was extended to the curved $n$--body problem on spaces of positive and negative
curvature in an unified way. In this paper the authors show that
the Lagrangian equilateral triangle solutions exist on spaces of constant
curvature if and only if the three masses are equal. 
Since then, a natural and pressing question emerged: Is it possible to extend
\emph{any} type of planar Newtonian relative equilibrium (curvature $\kappa = 0$) to
surfaces of constant curvature different from zero? In particular, what happens
with the Lagrangian solutions of the Newtonian problem when all masses are not
equal? Can they be continued to curved surfaces? 

Most researchers in the field though that this conjecture should be true,
but it has remained an open problem until today. In this paper
we finally present two proofs of this conjecture and a numerical continuation
procedure to compute relative equilibria for negative and positive curvature.

In recent years, relative equilibria on curved spaces have received
considerable attention; we now mention some related works.
For the case of
$n=3$ equal masses, the authors in \cite{Simo1} perform an extensive
analysis of the stability of the Lagrangian relative equilibria on spaces of
positive curvature, showing all possible bifurcation values and the regions on
the sphere where they are stable.  More generally, they study the homographic
orbits generated by this kind of relative equilibria.  In \cite{Simo2} the
same authors study relative equilibria of the restricted three-body problem,
when the two primary masses form a relative equilibrium.
For the case of $n=2$ masses, the authors in~\cite{Garcia}
perform an interesting analysis of the relative equilibria on spaces of
constant curvature using reduction and techniques from geometric mechanics,
and rediscover some results stated and proved first in \cite{GarciaN-Marrero},
but with simpler proofs.
In \cite{Er-Juan2} the authors show an inverse result for the collinear curved
$5$ and $7$ problem, where collinear means that all masses are located on the
same geodesic for all time. The same authors in \cite{Er-Juan1} show a local
and general regularization result for the restricted $n$--body problem, when the
$n$ primary equal masses form a relative equilibria.

Notice that all these previous works deal with particular cases, while our
present result applies to any non-degenerate relative equilibria (arbitrary number of bodies,
unequal masses, etc.). On the other hand, our result is perturbative and
applies only for small enough curvature.

The paper is organized as follows: In Section \ref%
{Section 2}, starting from a unified model for the curved $n$--body
problem, we express the corresponding Lagrangian in rotating coordinates and
obtain the kinetic and potential energy as an expansion in terms of the
curvature $\kappa$. In Section \ref{Section 3} we prove the continuation of
relative equilibria to spaces of constant curvature using Palais slice
coordinates.
In Section \ref{Section 4} we give an alternative proof using Lagrange
multipliers that is better suited for computations. Finally in Section
\ref{Section 5} we implement the continuation numerically and show several
examples, both for positive and negative
curvatures.

\section{Unified model for the $n$--body problem in spaces of constant
curvature}

\label{Section 2}

In this section we write the Lagrangian for the curved $n$--body problem for
positive and negative curvature $\kappa$%
, and we show that it can be written in a unified way for all $\kappa$
(positive, negative, and zero).

\subsection{Lagrangian of $n$ bodies on the sphere of positive curvature $%
\protect\kappa$}

From \cite{Pe}, the Lagrangian for $n$ bodies in the sphere of radius $R$,
where the position of the body with mass $m_{j}$ is parametrized by the
stereographic projection $z_{j}\in\mathbb{C}$, is given by 
\begin{equation}  \label{Lagrangian}
L(z,\dot{z};R)=T(z,\dot{z})+U(z),
\end{equation}
where the kinetic and potential energies are%
\begin{equation}
T=\frac{1}{2}\sum_{j=1}^{n}m_{j}\lambda(z_{j})\left\vert \dot{z}%
_{j}\right\vert ^{2},\qquad U=\sum_{j<k}m_{k}m_{j}V(z_{j},z_{k})\text{.}
\label{T-U}
\end{equation}
The function $\lambda$ is the conformal factor coming from the stereographic
projection through the north pole of the sphere into the plane of the ecuator,
and the potential $V$ corresponds to the mutual attraction among
the particles. Both functions are analytic in $z_{j}$ and $\bar{z}_{j}$
(i.e. analytic as a function of the real coordinates $x_{j}$ and $y_{j}$)
and given by%
\begin{align*}
V(z_{j},z_{k}) & =\frac{1}{R}\frac{\frac{4\Re\left( z_{j}\bar{z}_{k}\right) 
}{R^{2}}+\left( \frac{\left\vert z_{k}\right\vert ^{2}}{R^{2}}-1\right)
\left( \frac{\left\vert z_{j}\right\vert ^{2}}{R^{2}}-1\right) }{\sqrt{%
\left( \frac{\left\vert z_{k}\right\vert ^{2}}{R^{2}}+1\right) ^{2}\left( 
\frac{\left\vert z_{j}\right\vert ^{2}}{R^{2}}+1\right) ^{2}-\left( \frac{%
4\Re\left( z_{j}\bar{z}_{k}\right) }{R^{2}}+\left( \frac{\left\vert
z_{k}\right\vert ^{2}}{R^{2}}-1\right) \left( \frac{\left\vert
z_{j}\right\vert ^{2}}{R^{2}}-1\right) \right) ^{2}}}, \\
\lambda(z_{j}) & =\frac{4}{\left( 1+\frac{\left\vert z_{j}\right\vert ^{2}}{%
R^{2}}\right) ^{2}}\text{.}
\end{align*}

The positive curvature $\kappa$ depends on the radius $R$ as
\begin{equation*}
\kappa=\frac{1}{R^{2}}\text{.} 
\end{equation*}
Thus the functions $V$ and $\lambda$ expressed in terms of the curvature are 
\begin{align}
V(z_{j},z_{k}) & =\frac{4\Re\left( z_{j}\bar{z}_{k}\right) \kappa+\left(
\left\vert z_{k}\right\vert ^{2}\kappa-1\right) \left( \left\vert
z_{j}\right\vert ^{2}\kappa-1\right) }{\sqrt{\kappa^{-1}\left( \left\vert
z_{k}\right\vert ^{2}\kappa+1\right) ^{2}\left( \left\vert z_{j}\right\vert
^{2}\kappa+1\right) ^{2}-\kappa^{-1}\left( 4\Re\left( z_{j}\bar{z}%
_{k}\right) \kappa+\left( \left\vert z_{k}\right\vert ^{2}\kappa-1\right)
\left( \left\vert z_{j}\right\vert ^{2}\kappa-1\right) \right) ^{2}}},
\label{eq:Vpos}\\
\lambda(z_{j}) & =\frac{4}{\left( 1+\kappa\left\vert z_{j}\right\vert
^{2}\right) ^{2}}\text{.}
	\label{eq:lambdapos}
\end{align}

\subsection{Lagrangian of $n$ bodies on the hyperbolic sphere of negative
curvature $\protect\kappa<0$.}

From \cite{Pe2}, the Lagrangian for $n$ bodies in the hyperbolic sphere of
radius $R$, where the position of the body with mass $m_{j}$ is parametrized
by the stereographic projection $z_{j}\in\mathbb{C}$, is given by 
\begin{equation}
L(z,\dot z;R)=T(z, \dot z)+U(z),
\end{equation}
where the kinetic and potential energies are defined by (\ref{T-U}) with the
functions $\lambda$ and $V$ given by%
\begin{align*}
V(z_{j},z_{k}) & =\frac{\frac{-4\Re\left( z_{j}\bar{z}_{k}\right) }{R^{2}}%
+\left( -\frac{\left\vert z_{k}\right\vert ^{2}}{R^{2}}-1\right) \left( -%
\frac{\left\vert z_{j}\right\vert ^{2}}{R^{2}}-1\right) }{\sqrt{-R^{2}\left(
-\frac{\left\vert z_{k}\right\vert ^{2}}{R^{2}}+1\right) ^{2}\left( -\frac{%
\left\vert z_{j}\right\vert ^{2}}{R^{2}}+1\right) ^{2}+R^{2}\left( \frac{%
-4\Re\left( z_{j}\bar {z}_{k}\right) }{R^{2}}+\left( -\frac{\left\vert
z_{k}\right\vert ^{2}}{R^{2}}-1\right) \left( -\frac{\left\vert
z_{j}\right\vert ^{2}}{R^{2}}-1\right) \right) ^{2}}}, \\
\lambda(z_{j}) & =\frac{4}{\left( 1-\frac{\left\vert z_{j}\right\vert ^{2}}{%
R^{2}}\right) ^{2}}\text{.}
\end{align*}
Here, the stereographic projection through the north pole of the hyperbolic
sphere is used to move the problem to the Poincar\'e disk $\D_R^2$.

In this case the curvature is given by 
\begin{equation*}
\kappa=-\frac{1}{R^{2}}\text{.} 
\end{equation*}
Hence the functions $V$ and $\lambda$ expressed in terms of the curvature
have exactly the same form as in the case of positive curvature~\eqref%
{eq:Vpos}, \eqref{eq:lambdapos}.

\subsection{Unified model}

Notice that the expression of the Lagrangian coincides for positive and
negative curvature and is given by equation \eqref{eq:Vpos}. In the
following proposition we simplify the expression of the potential $V$.

\begin{proposition}
The function $V(z_{j},z_{k};\kappa)$ is analytic for $z_{j} \neq z_{k}$ and $%
\left\vert \kappa \right\vert <\kappa _{0}$ with $\kappa _{0}$ a positive
constant, because%
\begin{equation}  \label{VS}
V(z_{j},z_{k})=\frac{4\Re\left( z_{j}\bar{z}_{k}\right) \kappa+\left(
\left\vert z_{k}\right\vert ^{2}\kappa-1\right) \left( \left\vert
z_{j}\right\vert ^{2}\kappa-1\right) }{2\left\vert z_{j}-z_{k}\right\vert
\left( \left\vert z_{k}\right\vert ^{2}\left\vert z_{j}\right\vert
^{2}\allowbreak\kappa^{2}+2\Re\left( z_{j}\bar{z}_{k}\right)
\allowbreak\kappa+1\right) ^{1/2}}.
\end{equation}
\end{proposition}

\begin{proof}
By straightforward computations we have that 
\begin{align*}
& \left( \kappa\left\vert z_{j}\right\vert ^{2}+1\right) ^{2}\left(
\kappa\left\vert z_{k}\right\vert ^{2}+1\right) ^{2}-\left( 4\kappa
\Re\left( z_{j}\bar{z}_{k}\right) +\left( \kappa\left\vert z_{j}\right\vert
^{2}-1\right) \left( \kappa\left\vert z_{k}\right\vert ^{2}-1\right) \right)
^{2} \\
& =4\kappa\left\vert z_{j}-z_{k}\right\vert ^{2}\allowbreak\left( \left\vert
z_{k}\right\vert ^{2}\left\vert z_{j}\right\vert ^{2}\allowbreak\kappa
^{2}+2\Re\left( z_{j}\bar{z}_{k}\right) \allowbreak \kappa+1\right),
\end{align*}
so the function~\eqref{eq:Vpos} can be written as~\eqref{VS}. Thus $V$ is
analytic in $\kappa$ for $z_{j}\neq z_{k}$.
\end{proof}

Therefore we define the unified Lagrangian for $\kappa \in \mathbb{R}$ as $%
L(z,\dot{z};\kappa )=T(z,\dot{z})+U(z)$, where the kinetic energy $T$ and
potential energy $U$ are defined by (\ref{T-U}), and $V$ in (\ref{VS}).
Notice that when $\kappa =0$, we recover the planar Newtonian $n$--body
problem.

Next we express the Lagrangian in rotating coordinates $u_{j}\in\mathbb{C}$
for $j=1,\dotsc,n$. This will be useful for finding relative equilibria in
the following section.

\begin{proposition}
\label{Prop} Define the change of variables 
\begin{equation*}
z_{j}(t)=e^{it}u_{j}(t).
\end{equation*}%
Then the Lagrangian in rotating coordinates $u_{j}(s)$ is given by $L(u,\dot{%
u};\kappa )=T(u,\dot{u};\kappa )+U(u;\kappa )$, where%
\begin{align*}
T(u,\dot{u};\kappa )& =\frac{1}{2}\sum_{j=1}^{n}m_{j}\left\vert \left(
\partial _{t}+i\right) u_{j}\right\vert ^{2}\left( 4+\mathcal{O}(\kappa
^{2})\right) , \\
U(u;\kappa )& =\frac{1}{2}\sum_{j<k}m_{k}m_{j}\frac{1}{\left\vert
u_{k}-u_{j}\right\vert }+\mathcal{O}(\kappa )\text{.}
\end{align*}%
Here $\mathcal{O(}\kappa ^{n})$ are analytic functions of order $\kappa ^{n}$
that depend only on $u=(u_{1},...,u_{n})$.
\end{proposition}

\begin{proof}
The statement for the potential energy follows from the invariance under
rotations for expression (\ref{VS}), and by an expansion in power series of $%
\kappa $. The result for the Kinetic energy follows from the fact that $%
\lambda (z_{j})=\lambda \left( u_{j}\right) =4+\mathcal{O}(\kappa ^{2})$, and
that%
\begin{equation*}
\left\vert \partial _{t}z_{j}(t)\right\vert ^{2}=\left\vert e^{it}\left(
\partial _{t}u_{j}+iu_{j}\right) \right\vert ^{2}=\left\vert \left( \partial
_{t}+i\right) u_{j}\right\vert ^{2}.
\end{equation*}
\end{proof}

\section{Continuation of relative equilibria to spaces of constant curvature}

\label{Section 3}

In this section we state our main result. First we recall some basic aspects of
central configurations and relative equilibria.

\subsection{Central configurations in the plane}

Central configurations for the classical planar Newtonian $n$--body problem
are special positions of the particles where the position and acceleration
vectors of each particle are proportional, with the same constant of
proportionality for every particle. Central configurations play an important
role in celestial mechanics (see for instance \cite{Saari} and references
therein), but the main property of central configurations is 
that they generate the only known explicit solutions of the $n$--body
problem. These are known as homographic solutions and are given by 
\begin{equation*}
\gamma(t)=R(t)\Omega(\omega t)a, 
\end{equation*}
where $R(t)$ is a non-negative scalar function, $\Omega(\omega t)\in SO(2)$
is a rotation matrix with angular frequency $\omega$, and $a$ is a central
configuration.

A periodic solution of this form with $R(t) \equiv 1$ is called a
\emph{relative equilibrium}, since in a uniformly rotating frame it becomes an equilibrium point, or a steady solution of the corresponding flow.
Normalizing the constant of proportionality in the analysis of central
configurations, we define them as follows:

\begin{definition}
\label{def}A \textbf{central configuration} $a=(a_{1},\dots ,a_{n})\in 
\mathbb{R}^{2n}$ is a particular position of the particles on the plane which
	verifies the equations%
\begin{equation}
4a_{j}=\frac{1}{2}\sum_{k\neq j}m_{k}\frac{a_{j}-a_{k}}{\left\Vert
	a_{j}-a_{k}\right\Vert ^{3}}\qquad \textrm{for $j=1,\dotsc,n$.}  \label{cc}
\end{equation}
\end{definition}

\begin{remark}
\label{remark} If in Proposition \ref{Prop}, one instead makes the change the
	variables
\begin{equation*}
z_{j}(t)=e^{2^{-3/4}\omega it}u_{j}(t),
\end{equation*}%
where $\omega \in \mathbb{R}$ is a parameter, then the kinetic energy takes
the form 
\begin{equation}
\tilde{T}(u;\kappa )=\frac{1}{2}\sum_{j=1}^{n}m_{j}\frac{\omega ^{2}}{2}%
\lambda (u_{j})\left\vert u_{j}\right\vert ^{2}\text{,}
\end{equation}%
which follows from%
\begin{equation*}
\left\vert \partial _{t}z_{j}(t)\right\vert ^{2}=\left\vert \left( \partial
_{s}+2^{-3/4}\omega i\right) u_{j}\right\vert ^{2}.
\end{equation*}%
In this case, the Lagrangian is $L(u;\kappa )=\tilde{T}(u)+U(u)$ and a
critical point of $\nabla L(u;0)$ satisfies the classical definition of a
central configuration%
\begin{equation}
\omega ^{2} a_{j}=\sum_{k\neq j}m_{j}\frac{a_{j}-a_{k}}{\left\Vert
a_{j}-a_{k}\right\Vert ^{3}}.
\end{equation}
\end{remark}

In order to study relative equilibria as steady solutions in uniformly rotating
coordinates, we write the Lagrangian $L(u)=T(u)+U(u)$ in real coordinates $u\in
\mathbb{R}^{2n}$.
For complex variables $u_{j}=x_j + i y_j \in \mathbb{C}$, 
\begin{equation*}
\Re (u_{j}\bar{u}_{k})=x_{j}x_{k}+y_{j}y_{k},
\end{equation*}%
so the expressions \eqref{eq:Vpos}, \eqref{eq:lambdapos} for $\lambda $ and $V$
in real components are given by 
\begin{align}
V(u_{j},u_{k})& =\frac{4\left( u_{j}\cdot u_{k}\right) \kappa +\left(
\left\vert u_{k}\right\vert ^{2}\kappa -1\right) \left( \left\vert
u_{j}\right\vert ^{2}\kappa -1\right) }{2\left\vert u_{j}-u_{k}\right\vert
\left( \left\vert u_{k}\right\vert ^{2}\left\vert u_{j}\right\vert
^{2}\kappa ^{2}+2\left( u_{j}\cdot u_{k}\right) \kappa +1\right) ^{1/2}},
\label{RT} \\
\lambda (u_{j})& =\frac{4}{\left( 1+\kappa \left\vert u_{j}\right\vert
^{2}\right) ^{2}}\text{.}  \label{Rl}
\end{align}%
Let $\Omega \subset \mathbb{R}^{2n}$ be a given open collision-less
subset. For steady solutions, the Lagrangian obtained in Proposition %
\ref{Prop} in real components simplifies to 
\begin{align*}
L(u;\kappa )& :\Omega \subset \mathbb{R}^{2n}\times \mathbb{R}\rightarrow 
\mathbb{R}, \\
L(u;\kappa )& =\frac{1}{2}\sum_{j=1}^{n}4m_{j}\left\vert u_{j}\right\vert
^{2}+\frac{1}{2}\sum_{j<k}m_{k}m_{j}\frac{1}{\left\vert
u_{k}-u_{j}\right\vert }+\mathcal{O}(\kappa )\text{.}
\end{align*}

We define the action of the symmetry group $\theta\in G:=SO(2)$ in the space 
$u\in\mathbb{R}^{2n}$ according to 
\begin{equation*}
\theta\cdot u=e^{\mathcal{J\theta}}u~,
\end{equation*}
where $\mathcal{J}=J\oplus...\oplus J$ with 
\begin{equation}\label{eq:sympl_matrix}
J:=\left( 
\begin{array}{cc}
0 & -1 \\ 
1 & 0%
\end{array}
\right) \simeq i~. 
\end{equation}
It is easy to see that the Lagrangian action $L(u;\kappa)$ is invariant under
rotations, and the gradient is equivariant, that is:%
\begin{equation*}
L(\theta\cdot u;\kappa)=L(u),\qquad\nabla L(\theta\cdot u;\kappa)=\theta
\cdot\nabla L(u)\text{.} 
\end{equation*}

The gradient of $L(u;0)$ with respect to $u=(u_{1},...,u_{n})$ has components%
\begin{equation*}
\nabla _{u_{j}}L(u;0)=4m_{j}u_{j}\,-\frac{1}{2}\sum_{k=1(k\neq
j)}^{n}m_{j}m_{k}\frac{u_{j}-u_{k}}{\left\Vert u_{j}-u_{k}\right\Vert ^{3}}%
\text{.}
\end{equation*}%
Central configurations are critical points of the unperturbed gradient, $%
\nabla L(u;0)=0$. Therefore, the gradient $\nabla L(u;0)$ vanishes along all
the orbit of the central configuration $a$, 
\begin{equation*}
G(a)=\{\theta \cdot a=e^{\mathcal{J}\theta }a:\theta \in G\}\text{.}
\end{equation*}%
Next we show that these orbits of solutions persist when considering
the perturbation term $\mathcal{O}(\kappa )$ for small $\kappa $.

\subsection{Continuation to spaces of constant curvature using Palais slice
coordinates}

The variational formulation allows us to prove 
persistence of a relative equilibrium as the $SO(2)$-orbit of a central
configuration $a$ for small $\kappa$. For this purpose, we define the $\rho $%
-neighbourhood of radius $\rho$ around the group orbit $G(a)$, 
\begin{equation*}
\Omega=\{u\in\mathbb{R}^{2n}\mid\left\Vert u-\theta\cdot a\right\Vert
<\rho,~~\theta\in G\}\text{.} 
\end{equation*}

The orbit $G(a)$ is a differentiable manifold. The tangent component to the
orbit manifold $G(a)$ at $a$ is 
\begin{equation*}
\frac{d}{d\theta}\left( \theta\cdot a\right) _{\theta=0}=\frac{d}{d\theta }%
\left( e^{\mathcal{J}\theta}a\right) _{\theta=0}=\mathcal{J}a\text{.} 
\end{equation*}
Thus, the tangent space to the $G$-orbit at ${a}$ is $T_{u}G({a%
})=\{\lambda\mathcal{J}\in\mathbb{R}^{2n}:\lambda \in\mathbb{R}\}$%
. Denote by 
\begin{equation*}
W=\{u\in\mathbb{R}^{2n}:\left\langle u,\mathcal{J}a\right\rangle =0\}
\end{equation*}
the orthogonal complement to the tangent space of the $G$-orbit of $a$ in $%
\mathbb{R}^{2n}$.

\begin{definition}
We say that $a$ is a \textbf{non-degenerate central configuration} if the
Hessian $D_{u}^{2}L(a;0)$ has only one zero eigenvalue with eigenvector $%
\mathcal{J}a$ corresponding to the generator of the $SO(2)$-rotations.
\end{definition}

For example, consider one of Lagrange's equilateral triangle solutions
following rigid circular motion. The relative equilibrium generated by the
central configuration $a$ is linearly stable if the mass parameter 
\begin{equation*}
\beta=27(m_{1}m_{2}+m_{1}m_{3}+m_{2}m_{3})/(m_{1}+m_{2}+m_{3})^{2} 
\end{equation*}
is less than 1. When $\beta=1$ there is a loss of linear stability, which
means a zero eigenvalue, that is, the relative equilibrium is degenerate.
For these mass values, we can not extend the corresponding relative equilibrium
to spaces of constant curvature (for more details, see \cite{Sic, Routh, Gas}).


We state our main theorem:

\begin{theorem}
\label{Thm3}Assume that $a$ is non-degenerate central configuration in the
plane. Then there is a positive constant $\kappa _{0}$ such that for all $%
\left\vert \kappa \right\vert <\kappa _{0}$, the equation $\nabla L(u;\kappa
)=0$ has a $G$-orbit of solutions given by 
\begin{equation*}
u(\kappa )=a+\mathcal{O}(\kappa )\in \mathbb{R}^{2n}.
\end{equation*}
\end{theorem}

\begin{proof}
We prove this theorem following ideas from \cite{Fo}. By Palais' Slice
Theorem, the orbit $G(a)$ has an invariant neighborhood $\mathcal{U}$
diffeomorphic to $G\times W_{0}$, where $W_{0}$ is a $G$-invariant
neighborhood of $0\in W_{0}\subset W$. Specifically,\ there is a unique map $%
\upsilon :G\times W_{0}\rightarrow \mathcal{U}$ defined in a neighborhood $%
\mathcal{U}\subset W$ of the orbit $G(a)$ such that $\upsilon (\theta ,0)=e^{%
\mathcal{J}\theta }a$ for $(\theta ,0)\in G\times W_{0}$. The image $%
\mathcal{S}=\{\upsilon (0,w)\in \mathcal{U}:w\in W_{0}\}$ is called the
slice and $\mathcal{U}$ is called a tube of the orbit $G(a)$. The map $%
\upsilon $ provides slice coordinates $(\theta ,w)$ of $\mathcal{U}$ with $%
(\theta ,w)\in G\times W_{0}$. The Lagrangian $L$ defined in coordinates $%
(\theta ,w)$ is given by 
\begin{equation*}
\mathcal{L}_{\kappa }(\theta ,w):=L(\upsilon (\theta ,w);\kappa ).
\end{equation*}
Notice that $\mathcal{L}_{\kappa }(\theta ,w)$ is $G$-invariant under the
natural action of $G$ on $G\times W_{0}$, so $\mathcal{L}_{\kappa }(\theta
,w)$ does not depend on $\theta $, i.e. $\mathcal{L}_{\kappa
}(w):W_{0}\rightarrow \mathbb{R}.$ By hypothesis the Hessian $D_{w}^{2}%
\mathcal{L}_{0}(0)$ is invertible, so the implicit function theorem
implies that there is a $\kappa _{0}>0$ and a unique map $w:\{\left\vert
\kappa \right\vert <\kappa _{0}\}\rightarrow W_{0}$ satisfying $\nabla _{w}%
\mathcal{L}_{\kappa }(w(\kappa ))=0$ and $w(0)=0$. The fact that $\mathcal{L}%
_{\kappa }(w)$ is differentiable in $\kappa $ implies that $w(a)=\mathcal{O}%
(\kappa )$.
\end{proof}

Therefore, the equation $\nabla L(u;\kappa )=0$ defines implicitly a
relative equilibrium $u$ as a function of the curvature $\kappa $. In other
words, under the hypotheses of the theorem, one can continue a relative
equilibrium on the plane into a family of relative equilibria on spaces of
both positive and negative constant curvature.

Notice that we use the stereographic projection on the plane, so the space
is fixed and we do not need to constrain the masses to be on the curved
space (as the authors do in~\cite{DPS12}, for example). Of course, one can
always invert the projection to obtain the relative equilibria on the curved
space.

\section{Continuation to spaces of constant curvature using Lagrange
multipliers}

\label{Section 4}

In the following we give an alternative proof of Theorem \ref%
{Thm3} using Lagrange multipliers. This second proof has the advantage
that it can be easily implemented numerically.

We can obtain the numerical continuation of a given relative equilibrium $a$ to
spaces of constant curvature by considering the augmented system

\begin{equation*}
\nabla L(u)+\alpha \mathcal{J}u=0,
\end{equation*}%
where $\mathcal{J}u\,$is the generator field under rotations $SO(2)$, and $%
\alpha $ is a Lagrange multiplier or unfolding parameter. In this case we
implement the Poincar\'{e} section as the set $u\cdot e=0$ where $e$ is a
unitary vector. 

\begin{proposition}
The zeros of the augmented map
\begin{equation*}
F(u,\alpha ;\kappa )=\binom{\nabla L(u;\kappa )+\alpha \mathcal{J}u}{%
(u-a)\cdot e}:\mathbb{R}^{2n+1}\times \mathbb{R\rightarrow R}^{2n+1}
\end{equation*}%
are relative equilibria of the $n$-body problem in the space of constant
curvature $\kappa $.
\end{proposition}

\begin{proof}
This follows from Proposition \ref{Prop} and the fact that the system is
invariant under rotation of $SO(2)$, i.e. $\nabla L(u)\cdot \mathcal{J}u=0$
for any $u$. Indeed, we only need to prove that a zero of $F$ has $\alpha =0$%
, which follows from
\begin{equation*}
0=\left( \nabla L(u)+\alpha \mathcal{J}u\right) \cdot \mathcal{J}u=\alpha
\left\vert u\right\vert ^{2}\text{.}
\end{equation*}
\end{proof}

\begin{proposition}
\label{alter} Let $a$ be a non-degenerate relative equilibrium and $e\cdot 
\mathcal{J}a\neq 0$, then the augmented map $F$ has a unique continuation of solutions $%
\left( u;\kappa \right) $ starting from $\left( a;0\right) $.
\end{proposition}

\begin{proof}
We prove this following ideas in \cite{Do}. Since $a$ is a central
configuration, it satisfies $\nabla _{u}L(a;0)=0$, i.e. $F(a,0;0)=0$.
By the implicit function theorem it is enough to prove that 
\begin{equation*}
D_{(u,\alpha )}F(a,0;0)=\left( 
\begin{array}{cc}
D_{u}^{2}L(a;0)+\alpha \mathcal{J}u & \mathcal{J}u \\ 
e^{T} & 0%
\end{array}%
\right) :\mathbb{R}^{2n+1}\times \mathbb{R\rightarrow R}^{2n+1}
\end{equation*}%
is invertible at a non-degenerate relative equilibrium $a$ with $\alpha =0$
(see the previous proposition). Consequently, it is enough to prove that $%
D_{(u,\alpha )}F(a,0;0)(u,\alpha )=0$ implies that $(u,\alpha )=0$. Assume
that $u\neq 0$. Then the equation $D_{u}^{2}L(a)u+\alpha \mathcal{J}a=0$
implies that $D_{u}^{2}L(a)u\in \ker D_{u}^{2}L(a)$. But $\ker D_{u}^{2}L(a)$
is orthogonal to the range of $D_{u}^{2}L(a)$, because $D_{u}^{2}L(a)$ is
self adjoint, so $\alpha =0$ and $u\in \ker D_{u}^{2}L(a)$. But the
equation $e^{T}u=0$ implies that $u=0$ , since $e\cdot \mathcal{J}a\neq 0$.
\end{proof}

\section{Numerical continuation of relative equilibria}

\label{Section 5}The system for $u_{j}=x_{j}+iy_{j}$ in real coordinates $%
u_{j}=(x_{j},y_{j})\in \mathbb{R}^{2}$ is the following: the Lagrangian $%
L(u;\kappa )=T(u)+U(u)$ has kinetic and potential energies $T$ and $U$ with
formulas (\ref{T-U}), where the expressions for $\lambda $ and $V$ in real
components are presented in (\ref{RT}) and (\ref{Rl}). Thus in this section
we set the numerical continuation of non-degenerate relative equilibria $a\in \mathbb{R}^{2n}$
satisfying (\ref{cc}) as zeros of the augmented map
\begin{equation}
F(u,\alpha ;\kappa )=\binom{\nabla _{u}L(u,\bar{u})+\alpha \mathcal{J}u}{%
u\cdot \mathcal{J}a}:\mathbb{R}^{2n}\times \mathbb{R}^{2}\mathbb{\rightarrow
R}^{2n}\times \mathbb{R}\text{,}  \label{eq:RE_condition}
\end{equation}%
where $F(a,0;0)=0$,  $\mathcal{J}=J\oplus ...\oplus J$, and $J$ is the
symplectic matrix~\eqref{eq:sympl_matrix}.

\subsection{Gradient of the Lagrangian}\label{sec:Gradient_of_the_Lagrangian}

In this section we find an expression for the gradient $\nabla_{u} L(u)$ in real
coordinates $(x_{k}, y_{k})\in\mathbb{R}^{2}$.

The real derivatives of the kinetic energy $T$ are
\begin{align}
\pd{T}{x_k}  &= m_{k} \left(  \lambda(u_{k})\frac{-2
x_{k}}{R^{2}+\abs{u_k}^{2}} \abs{u_k}^{2} + \lambda(u_{k}) x_{k} \right) ,\\
\pd{T}{y_k}  &= m_{k} \left(  \lambda(u_{k})\frac{-2
y_{k}}{R^{2}+\abs{u_k}^{2}} \abs{u_k}^{2} + \lambda(u_{k}) y_{k} \right) .
\end{align}

The real derivatives of the potential energy $U$ are
\begin{align}
\pd{U}{x_k}  &= \sum_{j=1, j\neq k}^{n} m_{k}m_{j} \pd{V}{x_k},\\
\pd{U}{y_k}  &= \sum_{j=1, j\neq k}^{n} m_{k}m_{j} \pd{V}{y_k},
\end{align}
where
\begin{multline*}
\pd{V}{x_k}  =
\frac{1}{2}\frac{
\left( 4 x_j \kappa + 2 x_k\kappa \left( \abs{u_j}^2\kappa -1 \right) \right) 
\left( A(u_j, u_k)^{1/2} B(u_j, u_k)^{1/2} \right)
}{A(u_j,u_k)B(u_j,u_k)} \\
-\frac{1}{2}\frac{
\left( 4\left( u_{j}\cdot u_{k}\right) \kappa+\left(
\left\vert u_{k}\right\vert ^{2}\kappa-1\right) \left( \left\vert
u_{j}\right\vert ^{2}\kappa-1\right) \right) 
}{A(u_j,u_k)B(u_j,u_k)} \\
\cdot \left( -(x_j-x_k)A(u_j, u_k)^{-1/2} B(u_j, u_k)^{1/2} +  
\left( x_k\abs{u_j}^2\kappa^2 + x_j\kappa \right) A(u_j, u_k)^{1/2} B(u_j, u_k)^{-1/2} \right)
\end{multline*}
and
\begin{align*}
A(u_j,u_k) &= \abs{u_j-u_k}^2, \\
B(u_j,u_k) &= \abs{u_k}^2 \abs{u_j}^2 \kappa^2 + 2(u_j \cdot u_k) \kappa + 1.
\end{align*}

\subsection{Relative Equilibria on the Plane}

As stated in Definition \ref{def}, a central configuration verifies equations
\eqref{cc}. For equal masses, the polygonal configuration is central
\cite{GaIz11}:

\begin{proposition} \label{prop:polygonal_RE}
Let%
\begin{equation*}
s_{1}=\sum_{j=1}^{n-1}\frac{1-e^{ij\zeta }}{\left\Vert 1-e^{ij\zeta
}\right\Vert ^{3}}=\frac{1}{2^{2}}\sum_{j=1}^{n-1}\frac{1}{\sin (j\pi /n)}.
\end{equation*}%
In the case of $n$ equal masses of value $1$, we have the polygonal central configuration $%
a_{j}=re^{i\zeta }\in \mathbb{C}$ satisfying (\ref{cc}) with 
\begin{equation*}
r=\frac{1}{2}\left( s_{1}\right) ^{1/3}.
\end{equation*}
\end{proposition}

\begin{proof}
We have%
\begin{align*}
4a_{j}-\frac{1}{2}\sum_{k\neq j}\frac{a_{j}-a_{k}}{\left\Vert
a_{j}-a_{k}\right\Vert ^{3}} & =4a_{j}-\frac{1}{2}\sum_{k\neq j}\frac{%
a_{j}-a_{k}}{\left\Vert a_{j}-a_{k}\right\Vert ^{3}} \\
& =a_{j}\left( 4-\frac{1}{2}\frac{1}{r^{3}}\sum_{j=1}^{n-1}\frac {%
1-e^{ij\zeta}}{\left\Vert 1-e^{ij\zeta}\right\Vert ^{3}}\right) =a_{j}\left(
4-\frac{1}{2}\frac{1}{r^{3}}s_{1}\right) .
\end{align*}
\end{proof}

For $n=3$ equal masses of value $1$ we have 
\begin{equation*}
s_{1}=\frac{1}{4}\sum_{j=1}^{2}\frac{1}{\sin (j\pi /3)}=\allowbreak 3^{-1/2}
\end{equation*}%
and 
\begin{equation}
r=\frac{1}{2}\left( 3^{-1/2}\right) ^{1/3} \approx 0.416\,34\text{.}
\end{equation}
\subsection{Numerical Results}

\subsubsection{Equal Masses}

Perhaps the most famous relative equilibrium is Lagrange's equilateral
triangle, which was derived in Proposition~\ref{prop:polygonal_RE}.
Assuming three equal masses $m_1=m_2=m_2=1$, Lagrange's equilateral relative
equilibrium $(z_1, z_2, z_3)\in\C^3$ is
\[ a = (z_1, z_2, z_2) = (r e^{i 0}, r e^{i 2\pi/3}, r e^{i 4\pi/3}), \]
where 
$r=\frac{1}{2} 3^{-1/6}$. 

Using numerical methods one can continue Lagrange's equilateral triangle $a$
into a family of relative equilibria for positive curvatures.
We discretize the values of curvature as $\{\kappa_i = 0.01 i\}$ for
$i=0,1,2,\dotsc$.
For each curvature value $\kappa_i$, we apply an iterative multidimensional
root finding algorithm to solve equation~\eqref{eq:RE_condition}, using as seed
the previous relative equilibrium obtained (that of $\kappa_{i-1}$).
Detailed expressions for equation~\eqref{eq:RE_condition} were provided in
Section~\ref{sec:Gradient_of_the_Lagrangian}.
The root finding algorithm used is \texttt{gsl\_multiroot\_fdfsolver\_hybridsj}
from the GNU Standard Library.

Of course, we can also continue Lagrange's equilateral triangle $a$ for
negative curvatures in the exact same way.

\begin{figure}[htp]
	\centering
	\begin{subfigure}[b]{0.75\textwidth}
		\includegraphics[width=\textwidth]{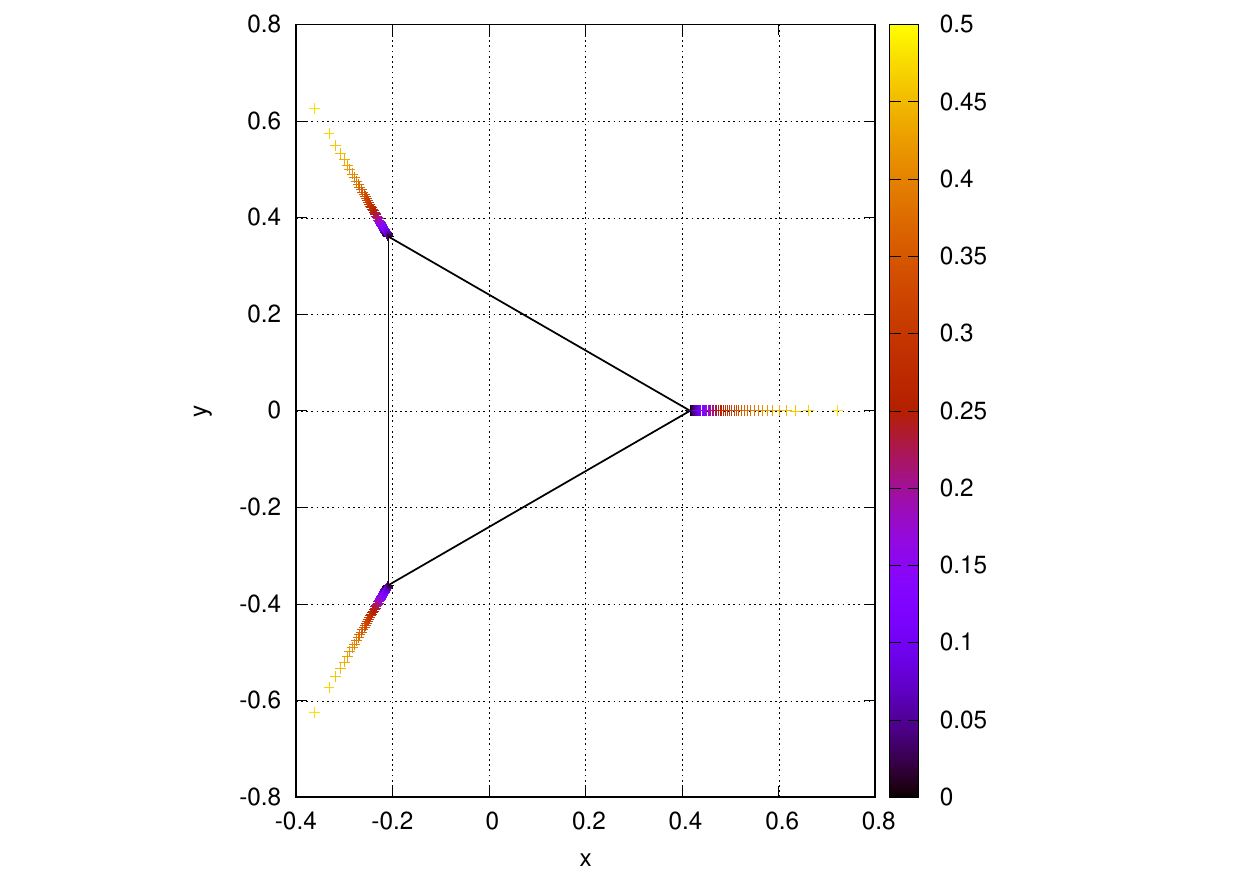}
		\caption{Positive curvatures.}
		\label{fig:RE_1_1_1_pos}
	\end{subfigure}
 
	\begin{subfigure}[b]{0.75\textwidth}
		\includegraphics[width=\textwidth]{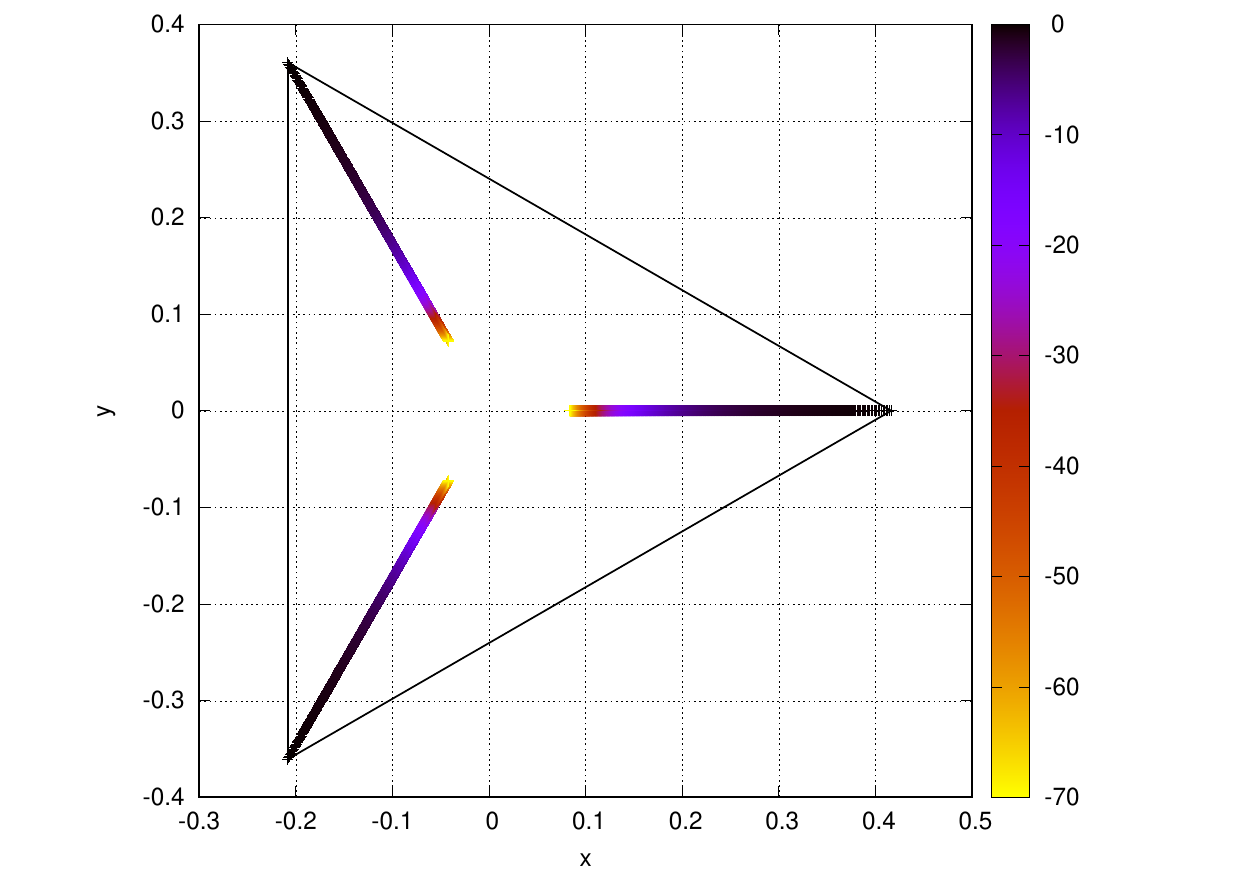}
		\caption{Negative curvatures.}
		\label{fig:RE_1_1_1_neg}
	\end{subfigure}

	\caption{Continuation of Lagrange's equilateral relative equilibrium $a$
	with equal masses $m_1 = m_2 = m_3 = 1$.
	Lagrange's equilateral relative equilibrium $a$ is marked with a black
	triangle.
	Relative equilibria continuating for curved spaces are color-coded
	according to their curvature value $\kappa$. 
	}
	\label{fig:RE_1_1_1}
\end{figure}

Using this methodology, we can numerically continue this relative
equilibrium to spaces of constant curvature $\kappa\in[-70, 0.48]$.
The family of relative equilibria obtained is shown in
Figure~\ref{fig:RE_1_1_1}. All configurations shown verify the relative
equilibrium condition~\eqref{eq:RE_condition} with precision $10^{-13}$.

When continuing to positive curvature, the family cannot be continued beyond
$\kappa\approx 0.48$, possibly due to a bifurcation: this is an interesting
question that we will try to answer in a forthcoming paper. When continuing to
negative curvature, on the other hand, the family can be continued up to
$\kappa\approx -70$ and beyond. 
We conjecture that the family actually extends to $\kappa \to -\infty$.
However, we loose numerical accuracy as the relative equilibrium approaches the
origin. In practice, we decided to stop our continuation when the numerical
accuracy falls below $10^{-13}$.

\subsubsection{Different Masses}

Consider now Lagrange's equilateral triangle on the plane, but assume three
different masses, for example $m_1=1$, $m_2=2$, $m_3=3$. It is not hard to find
this relative equilibrium analytically, given the following constrains: 
\begin{enumerate}
	\item The triangle is equilateral. 
	\item The length of each side must fulfill equation~\eqref{cc}.
	\item The center of mass is at the origin.
\end{enumerate}

\begin{figure}[htp]
	\centering
	\begin{subfigure}[b]{0.75\textwidth}
		\includegraphics[width=\textwidth]{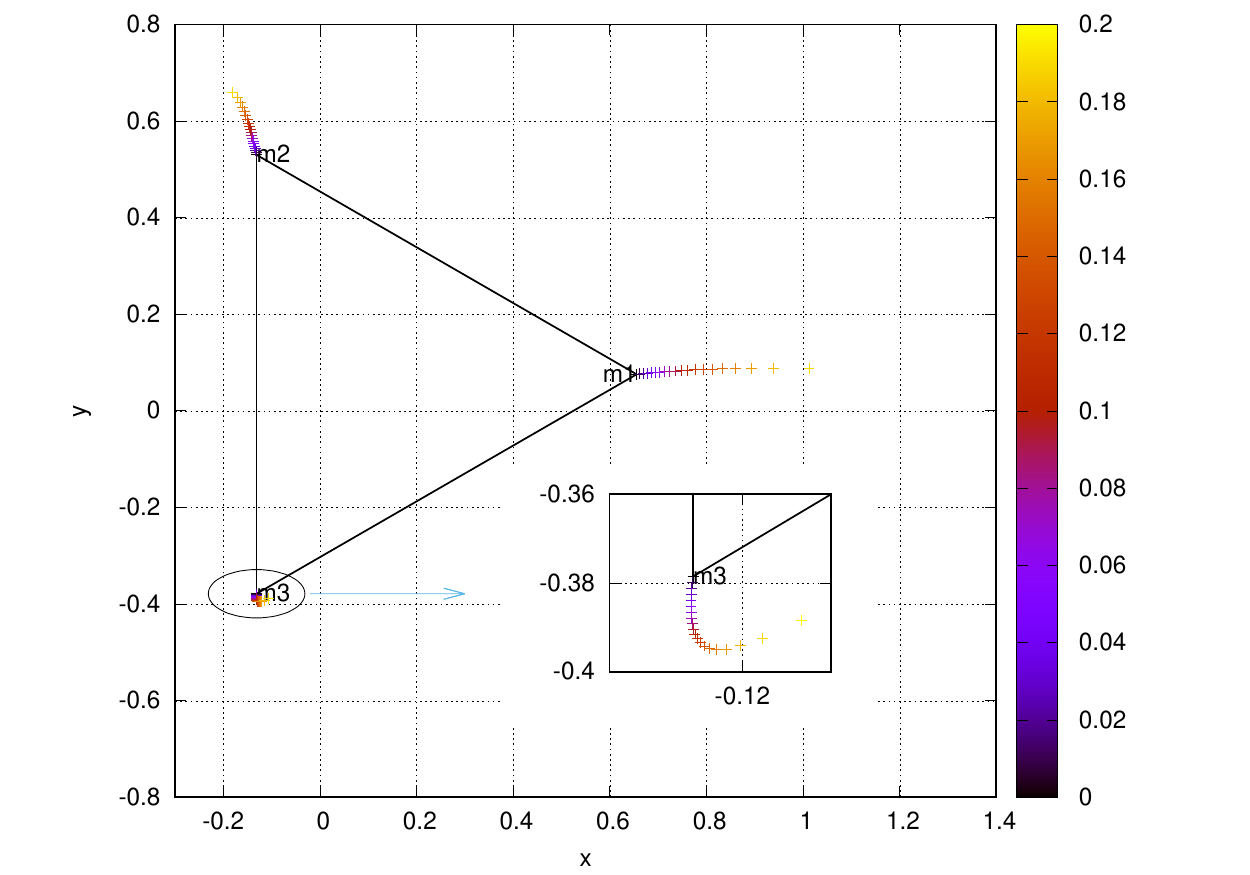}
		\caption{Positive curvatures. Inset: magnification of the region about
		$m_3$.}
		\label{fig:RE_1_2_3_pos}
	\end{subfigure}
 
	\begin{subfigure}[b]{0.75\textwidth}
		\includegraphics[width=\textwidth]{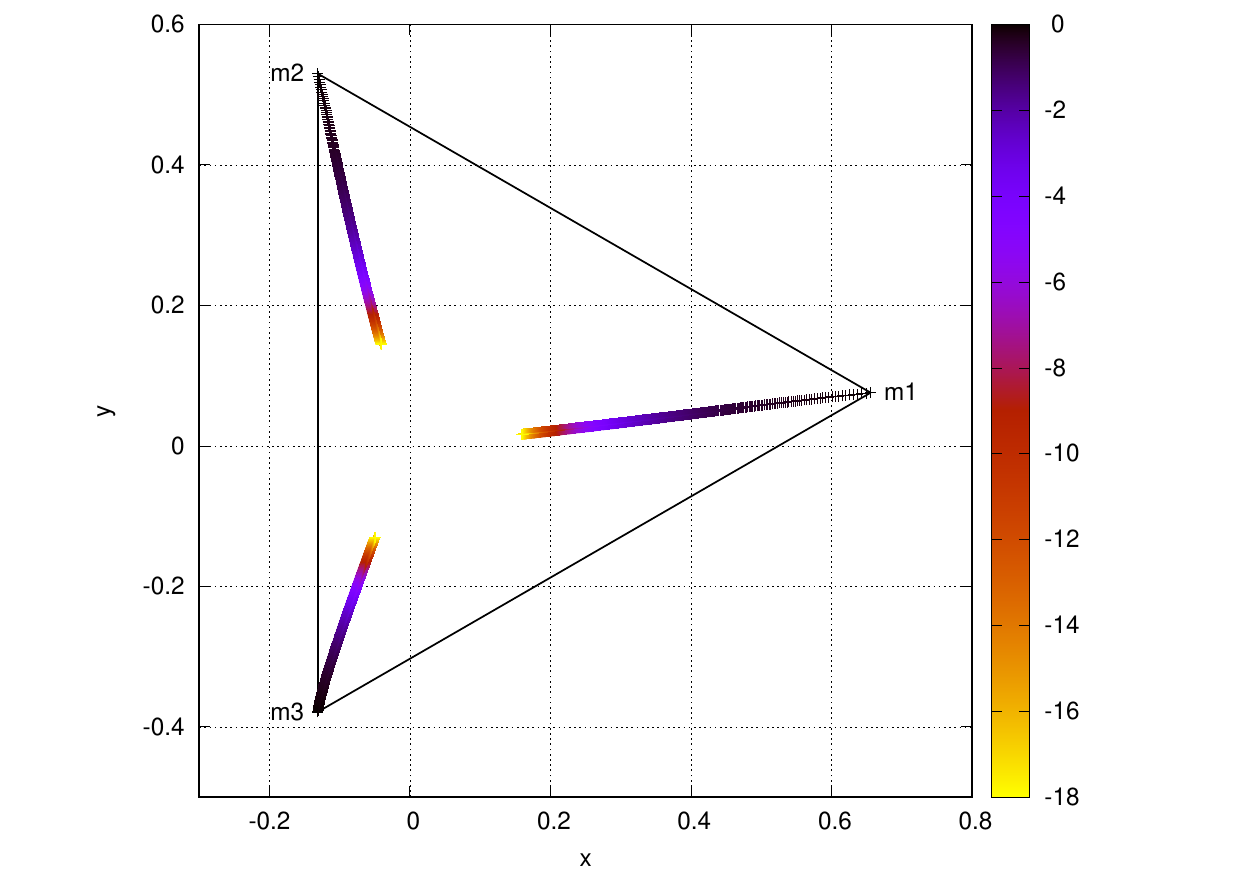}
		\caption{Negative curvatures.}
		\label{fig:RE_1_2_3_neg}
	\end{subfigure}

	\caption{Continuation of Lagrange's equilateral relative equilibrium $a$
	with different masses $m_1=1$, $m_2=2$, $m_3=3$.
	Lagrange's equilateral relative equilibrium $a$ is marked
	with a black triangle.
	Relative equilibria continuating for curved spaces are color-coded
	according to their curvature value $\kappa$. 
	}
	\label{fig:RE_1_2_3}
\end{figure}

Using the same methodology, we can numerically continue this relative
equilibrium to spaces of constant curvature $\kappa\in [-18, 0.19]$. 
The family of relative equilibria obtained is shown in
Figure~\ref{fig:RE_1_2_3}. All configurations shown verify the relative
equilibrium condition~\eqref{eq:RE_condition} with precision $10^{-13}$.

When continuing to positive curvature, the family cannot be continued beyond
$\kappa\approx 0.19$, possibly due to a bifurcation. When continuing to
negative curvature, on the other hand, the family can be continued up to
$\kappa\approx -18$ and beyond. 
We conjecture that the family actually extends to $\kappa \to -\infty$.
However, we loose numerical accuracy as the relative equilibrium approaches the
origin. In practice, we decided to stop our continuation when the numerical
accuracy falls below $10^{-13}$.

\subsubsection{Configurations in the sphere and hyperboloid}

The inverse transformation of the stereographic projection \cite{Pe,Pe2} defines the cartesian 
coordinates $(x,y,z)$ in the inertial frame of reference. The equation that must 
be fulfilled, due to the restriction that the particles move on the sphere or
hyperboloid, is $x^2+y^2+\sigma z^2 = \sigma R^2$, $\kappa = \sigma/R^2$, where $\sigma=1$ 
for the sphere, and $\sigma = -1$ for the hyperboloid.

In order to show all the configurations with the same curvature sign on a
common surface, we rescale the configurations. Configurations with curvature
$\kappa>0$ are rescaled in such a way that they lie on the surface of the
unitary sphere $x^2+y^2+z^2 = 1$. Configurations with
$\kappa<0$ are rescaled so that $x^2+y^2-z^2 = -1$. 

For positive curvature, the family that emerges from Lagrange's equilateral
triangle central configuration is shown in Figure \ref{fig:fampos}. 
For comparison we plot the case of equal masses $m_1=m_2=m_3=1$ and
different masses $m_1=1$, $m_2=2$, $m_3=3$ in the same figure. In a similar
way, the central configurations for negative curvature are shown in Figure
\ref{fig:famneg}.  

Under the rescaling, the families of central configurations share, in some
sense, features with the the classical case: for a given curvature $\kappa$, if
one body is more massive than other one then the body with more mass will be
closer to the axis $z$ (in the classical case $z=0$, and the point of reference
is the center of mass of the system), therefore also closer to the pole. This
property can be appreciated in Figures \ref{fig:orbpos}, \ref{fig:orbneg} for
positive and negative curvatures, respectively. We remark that this is the
opposite that happens without the rescaling. For instance, in the case of
positive curvature we have the relation $\kappa =1/R^2$, therefore $R \to 0$ as
$k \to \infty$.

As proved in~\cite{DPS12}, Lagrange's equilateral solutions exist on spaces of
constant curvature if and only if the three masses are equal and lie in a
circle parallel to the equator. We remark that, until now, there were
no results about relative equilibria where the three bodies have different masses
and move on circles with different heights $z$ \cite{Dia2011}.

\begin{figure}
	\centering
	\includegraphics[width=0.6\textwidth]{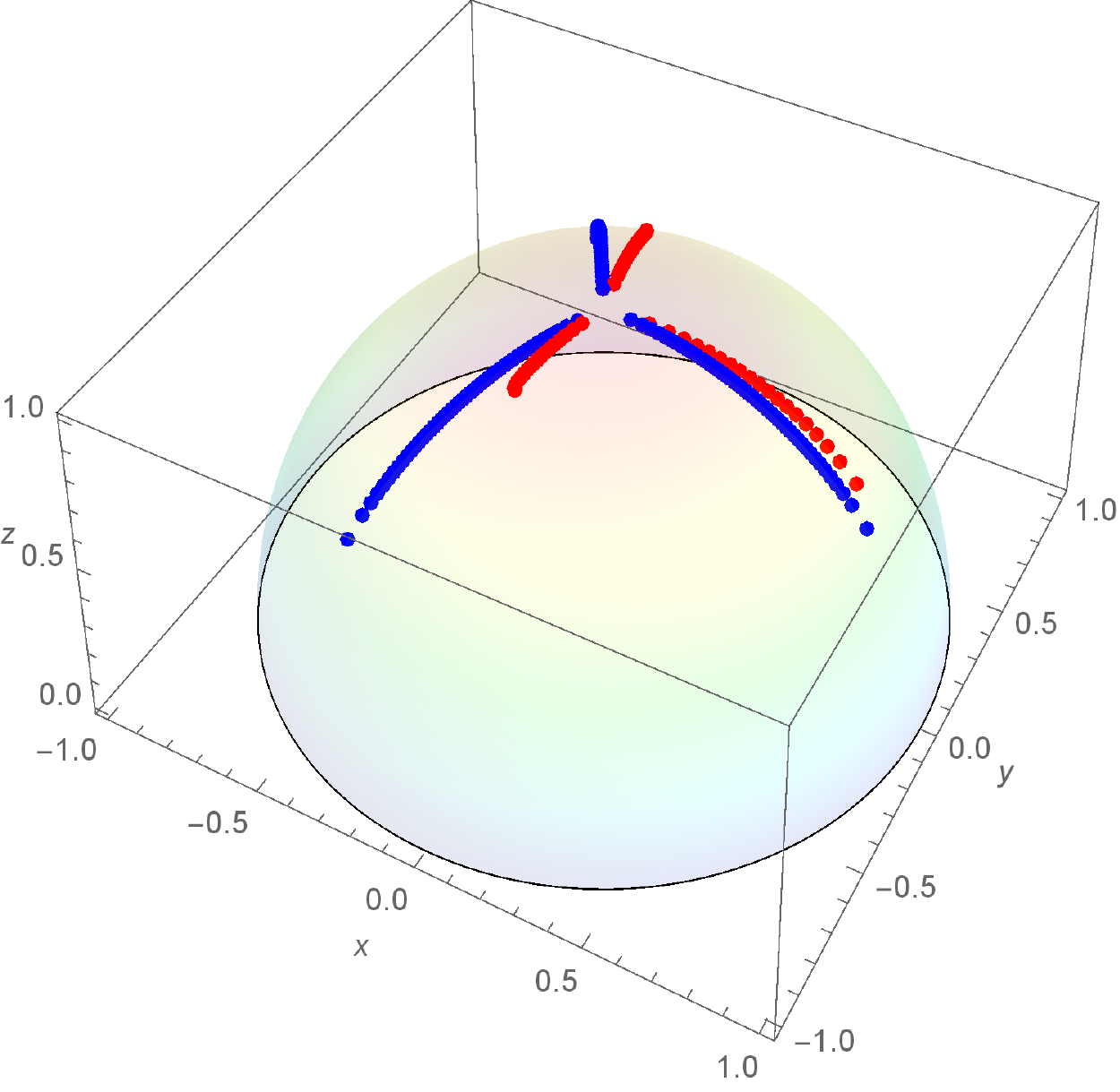}
	\caption{Continuation of Lagrange's relative equilibrium $a$ in cartesian 
	coordinates $(x,y,z)$ with 
	positive curvature $\kappa = 1/R^2$. Blue series for $m_1=1$, $m_2=1$, $m_3=1$, 
	and red series for $m_1=1$, $m_2=2$, $m_3=3$. The configuration is rescaled in
	such a way that $x_j^2+y_j^2+z_j^2 = 1$ for $j=1,2,3$, and the reflection
	$z \to -z$ is used.}
	\label{fig:fampos}
\end{figure}

\begin{figure}
	\centering
	\includegraphics[width=0.6\textwidth]{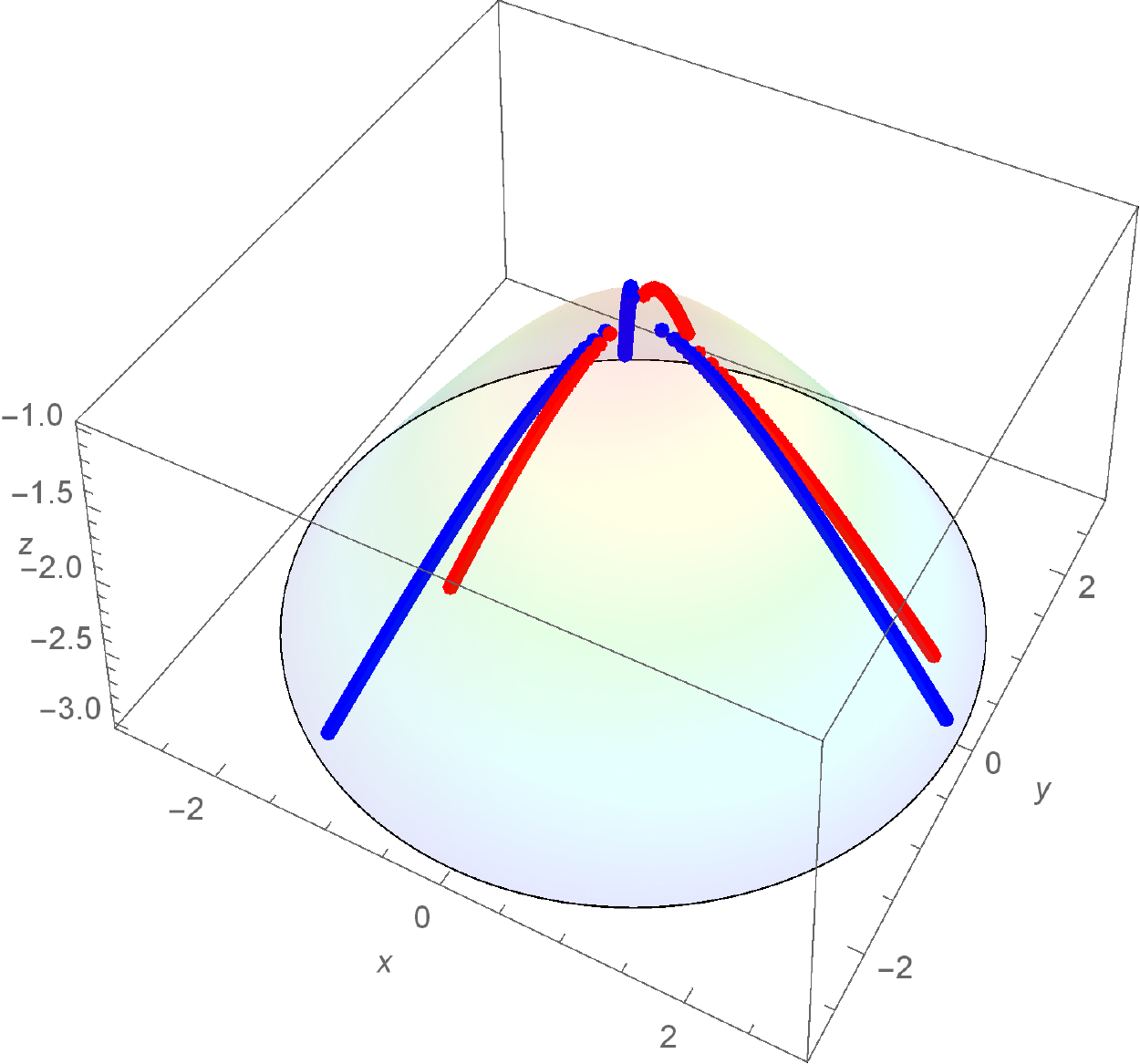}
	\caption{Continuation of Lagrange's relative equilibrium $a$ in cartesian 
	coordinates $(x,y,z)$ with 
	negative curvature $\kappa = -1/R^2$. Blue series for $m_1=1$, $m_2=1$, $m_3=1$, 
	and red series for $m_1=1$, $m_2=2$, $m_3=3$. The configuration is rescaled
	in such a way that $x_j^2+y_j^2-z_j^2 = -1$ for $j=1,2,3$, and the
	reflection $z \to -z$ is used.}
	\label{fig:famneg}
\end{figure}

\begin{figure}
	\centering
	\includegraphics[width=0.48\textwidth]{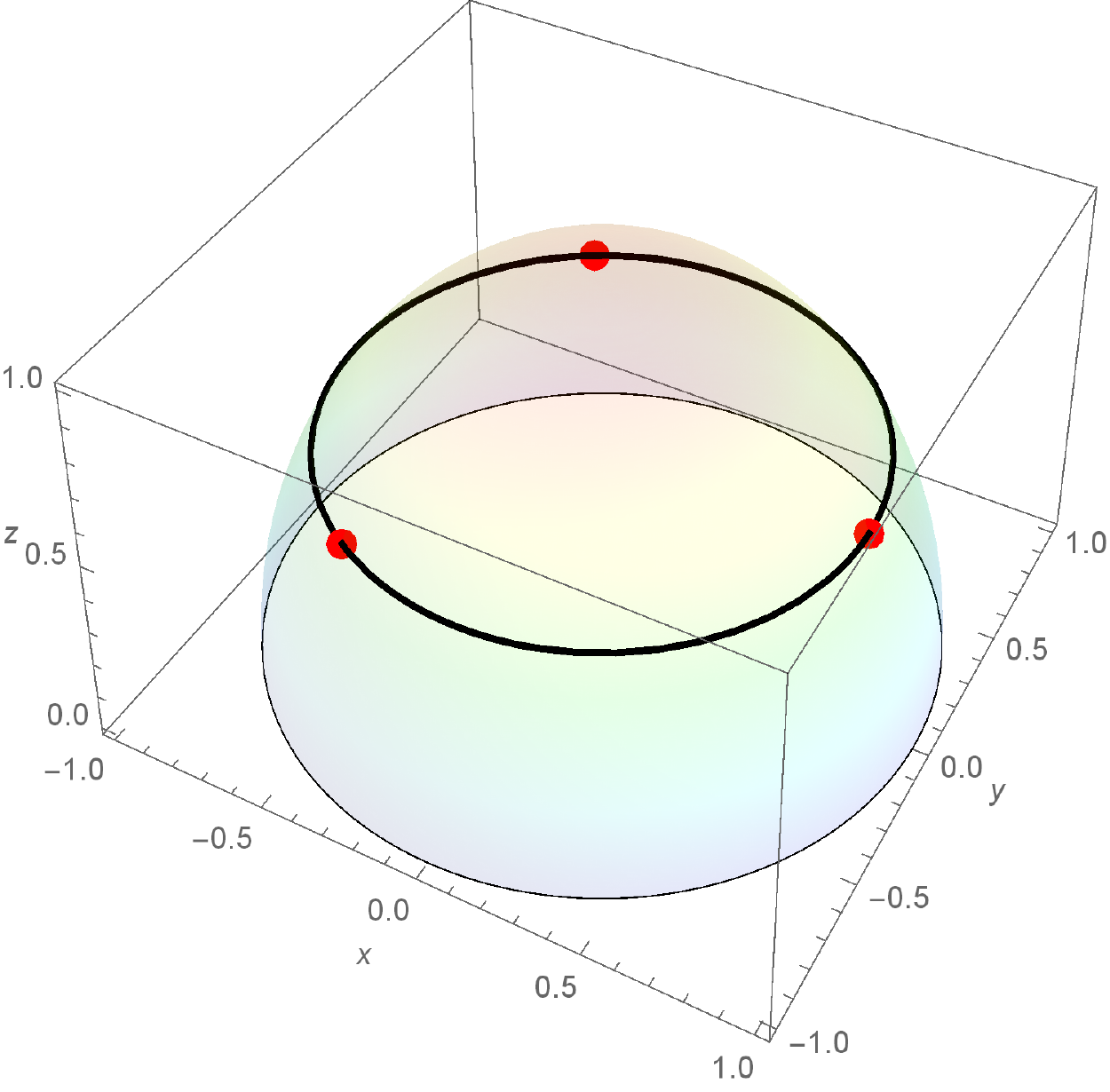}
	\includegraphics[width=0.48\textwidth]{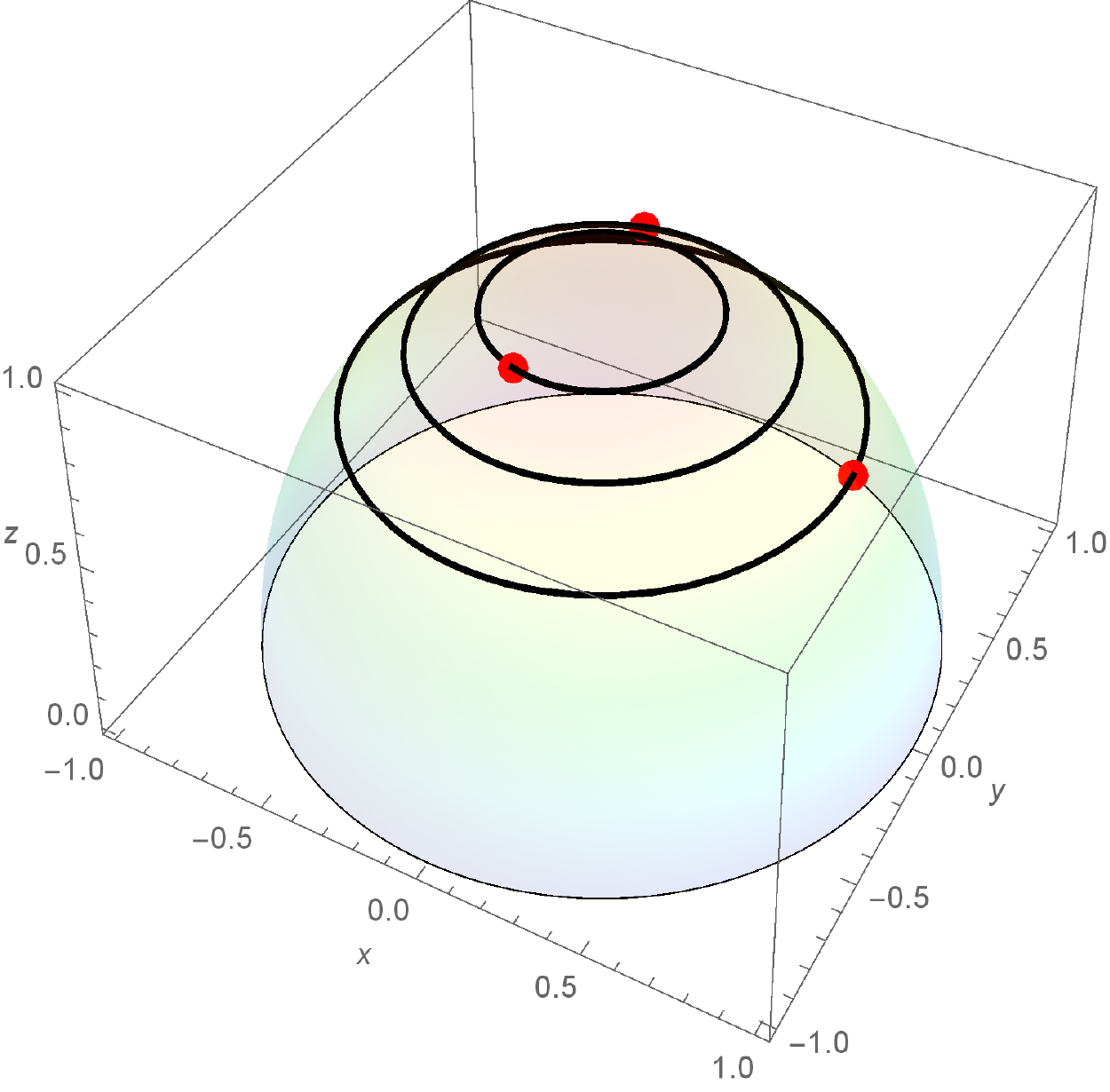}
\caption{Periodic circular orbits on the sphere. On the left the case $m_1=1$,
	$m_2=1$, $m_3=1$, with $\kappa = 0.49$, and on the right $m_1=1$, $m_2=2$,
	$m_3=3$, with $\kappa = 0.19$. Same rescaling and reflection as in Figure
	\ref{fig:fampos}.}
	\label{fig:orbpos}
\end{figure}

\begin{figure}
	\centering
	\includegraphics[width=0.48\textwidth]{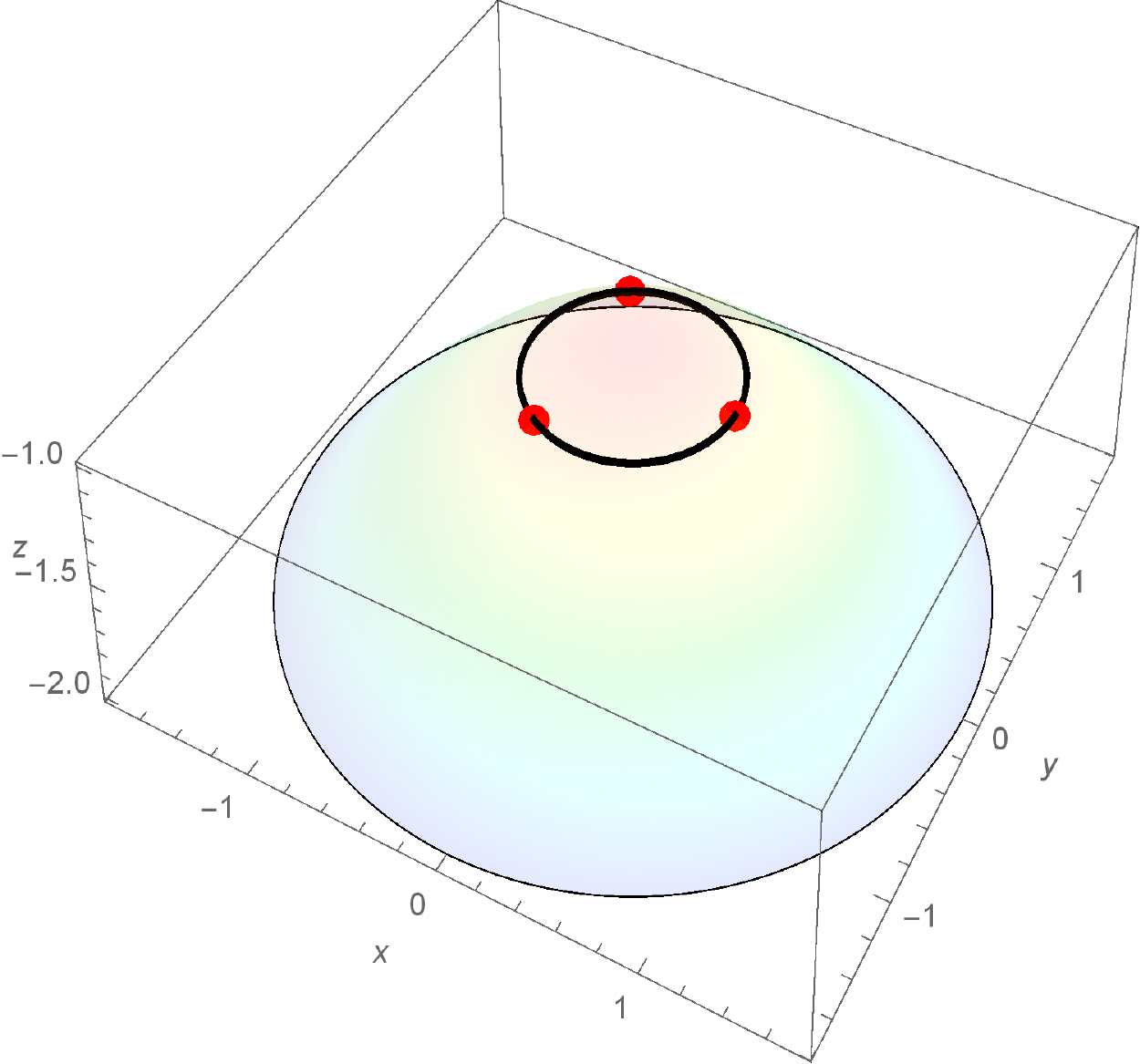}
	\includegraphics[width=0.48\textwidth]{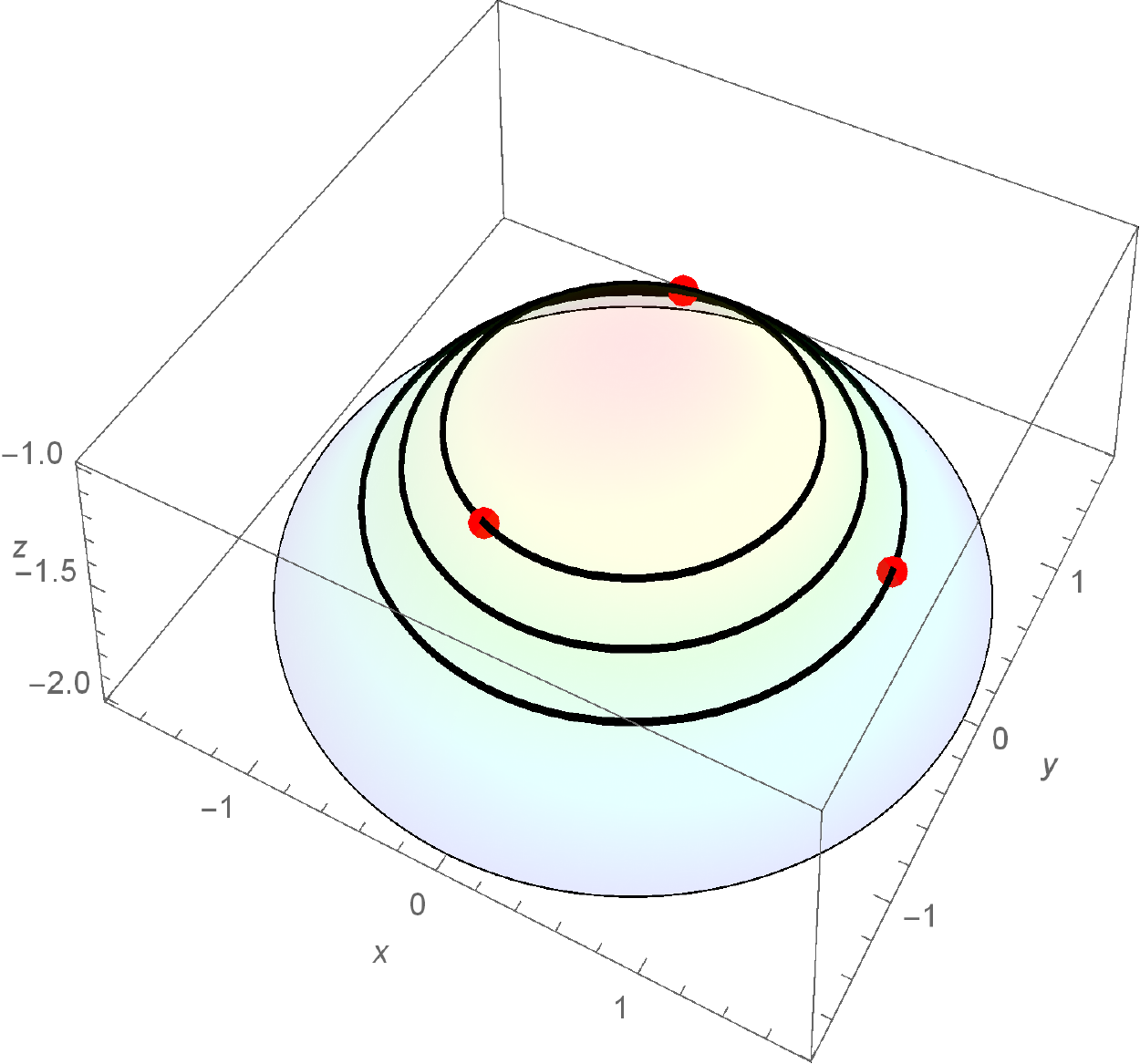}
\caption{Periodic circular orbits on the hyperboloid. On the left the case
	$m_1=1$, $m_2=1$, $m_3=1$, with $\kappa = -0.49$, and on the right $m_1=1$,
	$m_2=2$, $m_3=3$, with $\kappa = -1.99$. Same rescaling and reflection as
	in Figure \ref{fig:famneg}.}
	\label{fig:orbneg}
\end{figure}

\newpage

\subsection*{Acknowledgements}

The first and third authors have been partially supported by
\textit{Asociaci\'on Mexicana de Cultura A.C.} and by CONACYT, M\'exico,
project A1S10112. Research of P.R. was partially supported by NSF grant
DMS-1814543. C.G.A. was partially supported by UNAM-PAPIIT project IN115019.

\bibliography{ContinuationRE}
\bibliographystyle{plain}

\end{document}